\documentclass[a4paper,11pt]{amsart}
\usepackage[all]{xy}
\usepackage{graphicx,amsmath,amssymb,amsthm,amscd,mathrsfs}

\DeclareMathOperator{\pr}{pr}

\newtheorem{thm}{Theorem}[section]
\newtheorem{df}[thm]{Definition}
\newtheorem{prop}[thm]{Proposition}
\newtheorem{lem}[thm]{Lemma}
\newtheorem{cor}[thm]{Corollary}
\newtheorem{rem}[thm]{Remark}
\newtheorem{ass}[thm]{Assumption}

\begin{document}
\title[]{Index of rigidity of differential equations and Euler characteristic of 
their spectral curves}
\author{Kazuki Hiroe\\
	Chiba University
}
\email{kazuki@math.s.chiba-u.ac.jp}
\address{
Department of Mathematics and Informatics, Chiba University\\
1-33, Yayoi-cho, Inage-ku, Chiba-shi, Chiba, 263-8522 JAPAN}
\thanks{The author is supported by JSPS GGrant-in-Aid for Scientific Research (C)
Grant Number 20K03648.}

\keywords{Irregular singularity, Spectral curve,   
Index of rigidity, Euler characteristic  
Komatsu-Malgrange irregularity, Milnor number}
\subjclass{	14H70, 14H20, 34M35}

\begin{abstract}
	We show a coincidence of index of rigidity of differential equations with irregular singularities on a compact Riemann surface and
	Euler characteristic of the associated spectral curves which are recently called irregular spectral curves. Also we present a comparison of  local invariants, so called Milnor formula which
	links the Komatsu-Malgrange irregularity of differential equations and Milnor number of the spectral curves.
\end{abstract}
\maketitle

\section*{Introduction}
Higgs bundles with irregular singularities have been studied by many 
researchers from several points of view, mirror symmetry, geometric 
Langlands program, nonablelian Hodge correspondence and so on, see
\cite{BB}, \cite{KS}, \cite{Moc}, \cite{Sz}, \cite{Wit}.
At the same time, 
studies of relations between spectral curves and differential equations
recently increase in importance, see \cite{DHS}, \cite{DM} in which differential equations obtained as 
a quantization of spectral curves are discussed.

This paper presents a numerical comparison between cohomology groups of
a differential equation with irregular singularities on a Riemann
surface and those of  associated spectral curve.
A main result in this paper is the following.
We consider a differential equation $$dw=Aw$$ where
$A$ is a square matrix of size $n$ whose entries are meromorphic 1-forms on
a compact Riemann surface $X$ of genus $g$.
In particular this differential equation is allowed to have
several regular/irregular singular points $a_{1},a_{1},\ldots,a_{k}$ on $X$.
This differential equation defines a lambda connection $\nabla_{\lambda}$ on 
the trivial bundle $\mathcal{O}_{X}^{\oplus n}$ by
\[
	\nabla_{\lambda}=\lambda d-A	
\]
for $\lambda\in \mathbb{C}$, and we 
obtain a (possibly irregular) singular Higgs bundle 
$\nabla_{0}=A$ as a classical limit of the differential equation.
We can define a divisor on the cotangent bundle $T^{*}X$ as
the zero locus of the characteristic polynomial of the Higgs bundle $\nabla_{0}$,
\[
	\mathrm{det}(yI_{n}-A),
\]
and this is called spectral curve $C_{A}$.
Singular points appear as poles of $A$ and the zero locus of the
characteristic polynomial will pass through the line at infinity
$y=\infty$. Thus it is natural to consider the spectral curve
as a divisor on a compactified cotangent bundle $\overline{T^{*}X}$.
Let us assume one of the following conditions is satisfied at
each singular points $a_{1},a_{2},\ldots,a_{k}$.
\begin{enumerate}
\item The Hukuhara-Turrittin-Levelt normal form of the germ $A_{a_{i}}$
of $A$ at $a_{i}$ is multiplicity free (see Definition \ref{mfhtl}).
\item The germ $A_{a_{i}}$ is regular semisimple over $\mathbb{C}[\![z_{a_{i}}]\!]$ (see Definition \ref{rssdef}).
\end{enumerate}
Then we can show the following coincidence of the index of rigidity of
the differential equation and the Euler characteristic of the spectral
curve.
\begin{thm}[Theorem \ref{maintheorem}, Corollary \ref{cohomology}]
	Let $\nabla_{A}$ be the algebraic connection defined by the
	differential equation $dw=Aw$.
	Suppose that $C_{A}$ is irreducible. Moreover suppose that
	$C_{A}$ is smooth on $T^{*}X$.
	Then the index of rigidity $\mathrm{rig\,}(\nabla_{A})$ of $\nabla_{A}$ and
	the Euler characteristic $\chi(\widetilde{C_{A}})$ of the
	normalization $\widetilde{C_{A}}$ of $C_{A}$ coincide with
	each other, i.e.,
	\[
	\mathrm{rig}(\nabla_{A})=\chi(\widetilde{C_{A}}).
	\]
	Moreover assume that $\nabla_{A}$ is irreducible. Then we have the numerical coincidences of cohomology groups,
	\[
		h^{i}_{\mathrm{dR}}(X,j_{!*}(\mathcal{E}\mathrm{nd}\nabla_{A}))=h^{i}(\widetilde{C_{A}},\mathbb{C}),\quad
		i=0,1,2.
	\]
	Here $h^{i}(*):=\mathrm{dim}_{\mathbb{C}}H^{i}(*).$
\end{thm}
This fact has been known by Kamimoto \cite{Kam} and Oshima \cite{Osh}
for Fuchsian differential equations on $\mathbb{P}^{1}$.

Let us look at
the equation
\[
h^{1}_{\mathrm{dR}}(X,j_{!*}(\mathcal{E}\mathrm{nd}\nabla_{A}))=h^{1}(\widetilde{C_{A}},\mathbb{C})
\]
in our main theorem. This can be seen as an analogy of the well-known fact
 on the infinitesimal deformations of  a holomorphic Higgs bundle: the genus of the corresponding spectral curve  is equal to
  half of the dimension of the space of
the infinitesimal deformations, see \cite{Hit} and \cite{Nit}.
 That is to say, cohomology group $H^{1}_{\mathrm{dR}}(X,j_{!*}(\mathcal{E}\mathrm{nd}\nabla_{A}))$
 is known to be identified with the space of isotipical infinitesimal deformations
 of $\nabla_{A}$ by THEOREM 4.10 in \cite{BE} and also see Lemma 4.7 in
 \cite{Ar}. Here the isotipical deformation means the
 deformation of $\nabla_{A}$ under the condition that the HTL-normal forms at $a_{i}$, $i=1,2,\ldots,k$ are kept fixed.
Thus we may say that the main theorem gives an analogy of the fact for holomorphic
Higgs bundles to irregular meromorphic connections following the philosophy of
the nonabelian Hodge correspondence \cite{Sim}, \cite{BB}.

This main theorem is a consequence of the following local study of 
the singularities of the spectral curve.
As it is pointed out in \cite{KS} and \cite{Sz}, the spectral curve $C_{A}$ has intersections with
the line at infinity $X_{\infty}=\overline{T^{*}X}\backslash T^{*}X$ at
$\infty_{a_{i}}=(\infty,a_{i})$, $i=1,2,\ldots,k$
and these intersection points may have singularities resulting from the irregular
singularities of the corresponding differential equation and the Higgs bundle.
We investigate the singularities of the irregular spectral curve and show that the Milnor number of $C_{A}$ at $\infty_{a_{i}}$
can be computed from the Komatsu-Malgrange irregularity of the
corresponding local differential module $M_{A_{a_{i}}}$ as follows.

\begin{thm}[Theorem \ref{MILNER}]
	The Milnor number of $C_{A}$ at $\infty_{a_{i}}$ for each $i=1,2,\ldots,k$
	is
	\[
	\mu(C_{A})_{\infty_{a_{i}}}=-\delta(\mathrm{End}_{\mathbb{C}(\!\{z_{a_{i}}\}\!)}(M_{A_{a_{i}}}))
	-r_{C_{A_{a_{i}}}}
	+2(n-1)(C_{A},X_{\infty})_{\infty_{a_{i}}}+1.
	\]
\end{thm}
Here $(C,C')_{\infty_{a}}$ is the intersection number of divisors $C, C'$
at $\infty_{a}$, $r_{C_{A_{a_{i}}}}$ is the number of branches of the germ $C_{A_{a_{i}}}$
and
\begin{align*}
\delta(\mathrm{End}_{\mathbb{C}(\!\{z\}\!)}(M_{A_{a}}))&=
	\mathrm{rank\,}\mathrm{End}_{\mathbb{C}(\!\{z_{a}\}\!)}(M_{A_{a}})+\mathrm{Irr}(\mathrm{End}_{\mathbb{C}(\!\{z_{a}\}\!)}(M_{A_{a}}))\\
	 &\quad -\mathrm{dim\,}_{\mathbb{C}}\mathrm{Hom}_{\mathbb{C}(\!(z_{a})\!)}(\mathrm{End}_{\mathbb{C}(\!(z_{a})\!)}(\widetilde{M_{A_{a}}}),
	\mathbb{C}(\!(z_{a})\!))^{\text{hor}}.
\end{align*}
By using the $\delta$-invariant of a singularity of a plane curve germ, this formula can be written in
a simpler form.
\begin{cor}[Remark \ref{Delta}]
	\[
	2\delta(C_{A_{a_{i}}})-2(n-1)(C_{A},X_{\infty})_{\infty_{a_{i}}}=-\delta(\mathrm{End}_{\mathbb{C}(\!\{z_{a_{i}}\}\!)}(M_{A_{a_{i}}})).
	\]
\end{cor}
\subsection*{Acknowledgements}
The author would like to express his gratitude to Professors Shingo Kamimoto and  Toshio Oshima who gave him many inspirations through fruitful discussions.
The essential part of the idea to prove the main theorem
is based on their pioneering works in the case of Fuchsian differential equations on $\mathbb{P}^{1}$. The most part of this work had been done when
the author was a member of the Department of Mathematics in Josai University
and the work would have never been completed without the support from Josai University.
Finally the author would like to thank Professor Akane Nakamura.
Many discussions with her were very much inspiring and encouraging.
\section{Spectral curves of differential equations}\label{spectral curve}
\subsection{Compactified cotangent bundle on a Riemann surface}
Let $X$ be a compact Riemann surface of genus $g$ and
 consider a compactification of $T^{*}X$ defined by
\[
\overline{T^{*}X}:=\mathbb{P}(\mathcal{O}_{X}\oplus T^{*}X)
\]
which is the projective bundle of the vector bundle $\mathcal{O}_{X}\oplus T^{*}X$. The complement of $T^{*}X$ is denoted by  $X_{\infty}:=\overline{T^{*}X}\backslash
T^{*}X$.
The natural projection $$\pi\colon \overline{T^{*}X}\rightarrow X$$ enables us to
regard this surface as a ruled surface.
Thus the Neron-Severi group $\mathrm{NS\,}\overline{T^{*}X}:=\mathrm{Pic\,}\overline{T^{*}X}/\mathrm{Pic\,}^{0}\overline{T^{*}X}$ is generated by $X_{0}$, the
zero section of $T^{*}X$ and a fiber $f$,
namely,
\[
\mathrm{NS\,}\overline{T^{*}X}\cong \mathbb{Z}X_{0}\oplus \mathbb{Z}f.
\]
This lattice has the $\mathbb{Z}$-bilinear form determined by the intersection
numbers of generators,
$$
(X_{0},f)=1,\quad (f,f)=0,\quad (X_{0},X_{0})=-\mathrm{deg\ }T^{*}X=2g-2.$$

\subsection{Spectral curve of differential equation}
Let $$D:=\{a_{1},a_{2},\ldots,a_{p}\}\subset X$$ be a finite set.
We consider a differential equation with poles on $D$,
\[
dw=Aw
\]
where $A\in M(n,\Omega_{X}(*D)(X))$. We may assume that the set of all
poles of $A$ is exactly $D$.

Let us define the spectral curve of this differential equation through the 
notion of the \emph{lambda connection} as below.	
For $\lambda\in \mathbb{C}$, $\nabla_{\lambda}:=\lambda d-A$ defines a
	lambda connection on the trivial bundle $\mathcal{O}_{X}^{\oplus n}$,
	i.e., a $\mathbb{C}$-linear map $$\nabla_{\lambda}\colon \mathcal{O}_{X}^{\oplus n}\rightarrow \mathcal{O}_{X}^{\oplus n}\otimes \Omega_{X}(*D)$$
	satisfying $$\nabla_{\lambda}(fv)=\lambda df\otimes v+f\otimes \nabla_{\lambda}v$$
	for $f\in \mathcal{O}_{X}, v\in \mathcal{O}_{X}^{\oplus n}.$
	Then the \emph{spectral curve} $C_{A}$ is a divisor on $\overline{T^{*}X}$ 
	defined as the zero locus of the charachterisitic polynomial
	of the  Higgs bundle
	$(\mathcal{O}_{X}^{\oplus n},\nabla_{0})$ in the following way.
First let us take a complex atlas $\{(U_{i},z_{i})\}_{i=1,2,\ldots,k}$
of $X$ with an open covering $X=\bigcup_{i=1}^{k}U_{i}$ and
local coordinates $z_{i}\colon U_{i}\rightarrow \mathbb{C}$.
On this local coordinate system
the canonical 1-form $\theta\in \Omega_{T^{*}X}(T^{*}X)$ can be expressed as
$\theta=\eta_{i} dz_{i}$ with the fiber coordinate $\eta_{i}$
of $T^{*}X$ on $U_{i}$  for  $\ i=1,2,\ldots,k$.
Then the canonical 1-form $\theta$ extends to the meromorphic 1-form
$\bar\theta$ on $\overline{T^{*}X}$ of the form $$\bar\theta=\frac{\eta_{i}}{\zeta_{i}}dz_{i},\quad
i=1,2,\ldots,k,$$where $[\zeta_{i}\colon \eta_{i}]\in \mathbb{P}^{1}$ is
the fiber coordinate of $\overline{T^{*}X}=\mathbb{P}(\mathcal{O}_{X}\oplus T^{*}X)$
 on $U_{i}$.

Let us denote the trivialization of $A$ on $U_{i}$ by $A=A_{i}(z_{i})dz_{i}$,
$A_{i}(z_{i})\in M(n,\mathbb{C}(z_{i}))$.
Then the pullback $\pi^{*}A$ by the projection $\pi\colon \overline{T^{*}X}\rightarrow X$
can be written in the same form
\[
\pi^{*}A([\zeta_{i}:\eta_{i}],z_{i})=A_{i}(z_{i})dz_{i}
\]
on $\pi^{-1}(U_{i})\cong \mathbb{P}^{i}\times U_{i}$.

Then $\mathrm{det}\left(\frac{\eta_{i}}{\zeta_{i}}I_{n}-A_{i}(z_{i})\right)$
gives an meromorphic function of $\pi^{-1}(U_{i})$ for each
$i=1,2,\ldots, k$. Since the compatibility  follows immediately from the definition,
 the collection
$$\left\{
\left(
	\pi^{-1}(U_{i}),\,\mathrm{det}\left(\frac{\eta_{i}}{\zeta_{i}}I_{n}-A_{i}\right)
\right)
\right\}_{i=1,2,\ldots,k}$$
defines a Cartier divisor
on $\overline{T^{*}X}$.
The corresponding Weil divisor is the 
\emph{spectral curve} of the differential equation $dw=Aw$ and denoted by

\[
C_{A}\subset \overline{T^{*}X}.
\]
\subsection{Arithmetic genus of spectral curve}
We denote the divisor class of the spectral curve $C_{A}$ by the same notation.
The arithmetic genus $g_{a}(C_{A})$ of $C_{A}$ can be obtained by the genus formula
\[
	g_{a}(C_{A})=\frac{1}{2}(C_{A},C_{A}+K_{\overline{T^{*}X}})+1.
\]
Our complex surface $\overline{T^{*}X}$ is a ruled surface of which
the Neron-Severi group is well-understood. Thus  standard argument
enables us to examine the explicit value of  $g_{a}(C_{A})$,
see V.2 in \cite{Ha} and \cite{DM} for example.

Let us first determine the coefficients $a,\,b$ in the expression $C_{A}=aX_{0}+bf\in \mathrm{NS\,}\overline{T^{*}X}$.
Since the projection $\pi|_{C_{A}}\colon C_{A}\rightarrow X$ is of degree $n$,
we have
\[
(C_{A},f)=n.
\]
Thus
\[
n=(C_{A},f)=(aX_{0}+bf,f)=a.
\]
Next we note that $(X_{\infty}, X_{0})=0$
and $(X_{\infty},f)=1$.
This shows that
\[
b=(nX_{0}+bf)X_{\infty}=(C_{A},X_{\infty})
\]
and we have
\[
C_{A}=nX_{0}+(C_{A},X_{\infty})f\in \mathrm{NS\,}\overline{T^{*}X}.
\]

Also note that
\[
K_{\overline{T^{*}X}}=-2X_{0}+(2g-2+\mathrm{deg\,}T^{*}X)f=-2X_{0}+(4g-4)f.
\]
Finally, the genus formula leads us to
\begin{equation}\label{arithmetic}
	\begin{split}
g_{a}(C_{A})&=\frac{1}{2}(nX_{0}+(C_{A},X_{\infty})f,(n-2)X_{0}+((C_{A},X_{\infty})+4g-4)f)+1\\
&=\frac{1}{2}(n^{2}(2g-2)+(2n-2)(C_{A},X_{\infty}))+1.
	\end{split}
\end{equation}

\section{Local formal theory on differential equations}
Here we recall the Hukuhara-Turrittin-Levelt theory on local structure of
differential equations and the notion of irregularity introduced
by Komatsu \cite{Kom} and Malgrange \cite{Mal}.
\subsection{Differential modules over differential fields}
First let us fix notation. Let $\mathbb{C}[\![z]\!]$ and  $\mathbb{C}(\!(z)\!)$
denote the ring of formal power series and the field of formal Laurent series
respectively. Similarly $\mathbb{C}\{z\}$ and $\mathbb{C}(\!\{z\}\!)$
denote the ring of convergent power series and the field of convergent Laurent series.
Let $\mathcal{P}:=\bigcup_{s\in \mathbb{Z}_{>0}}\mathbb{C}(\!(z^{\frac{1}{s}})\!)$
be the field of Puiseux series. Also $\mathcal{P}^{\text{conv}}$ denote the
field of convergent Puiseux series.
Set $\mathcal{P}^{+}:=\bigcup_{s\in \mathbb{Z}_{>0}}\mathbb{C}[\![z^{\frac{1}{s}}]\!] $ $\mathcal{P}^{-}:=\bigcup_{s\in\mathbb{Z}_{>0}}z^{-\frac{1}{s}}\mathbb{C}[z^{-\frac{1}{s}}].$
Then we can decompose
\[
\mathcal{P}=\mathcal{P}^{-}\oplus \mathcal{P}^{+}.
\]
The order of  $f(z)=\sum_{r\in \mathbb{Q}}a_{r}z^{r}$ is the number
$$\mathrm{ord }f(z):=\mathrm{min}\{r\in \mathbb{Q}\mid a_{r}\neq 0\}.$$
Similarly, the order of Puiseux series
$G(z)=\sum_{r\in \mathbb{Q}}G_{r}z^{r}\in \mathcal{P}\otimes_{\mathbb{C}}M(n,\mathbb{C})$ with matrix coefficients is defined
by $\mathrm{ord\ }G(z):=\mathrm{min}\{r\in\mathbb{Q}\mid G_{r}\neq \mathbf{0}_{n}\}$
where $\mathbf{0}_{n}$ is the zero matrix of size $n$.

Let $\mathcal{K}$ be one of the following fields: $\mathbb{C}(\!(z)\!)$,
$\mathbb{C}(\!\{z\}\!)$,
 $\mathbb{C}(\!(z^\frac{1}{s})\!)$, $\mathbb{C}(\!\{z^{\frac{1}{s}}\}\!)$, $\mathcal{P}$ and $\mathcal{P}^{\text{conv}}$.
A differential module $M$ over $\mathcal{K}$ is a $\mathcal{K}$-module
with the derivation $\nabla_{M}\in \mathrm{End}_{\mathbb{C}}(M)$
satisfying the Leibniz rule $\nabla_{M}(km)=\frac{d}{dz}k\cdot m+k\cdot\nabla_{M}m$
for $k\in \mathcal{K}$ and $m\in M$.
Suppose that  $M$ is finite of rank $n$ over $\mathcal{K}$ and choose a basis $\mathbf{e}=\{e_{1},e_{2},\ldots,e_{n}\}$.
Then
the  matrix $G={}^{t}(g_{i,j})_{1\le i,j\le n}$
defined by
\[
\nabla_{M}e_{i}=\sum_{j=1}^{n}g_{i,j}e_{j}
\]
gives the  matrix form of $\nabla_{M}\in \mathrm{End}_{\mathbb{C}}(M)$,
that is
\[
	\frac{d}{dz}-G\in \mathrm{End}_{\mathbb{C}}(\mathcal{K}^{\oplus n}).
\]
We call $G$ the \emph{matrix of $\nabla_{M}$ with respect to $\mathbf{e}$}.
Conversely, $G\in M(n,\mathcal{K})$ defines a differential module
$M_{G}:=\mathcal{K}^{\oplus n}$ with the derivation $\nabla_{M_{G}}:=
\frac{d}{dz}-G$.

For two matrices $G,\,G'$ of $M$, there exists a base change matrix $X\in
\mathrm{GL}(n,\mathcal{K})$ and we have
\[
G'=XGX^{-1}+\left(\frac{d}{dz}X\right)X^{-1}.
\]

Let us recall some operations on finite differential modules.
For differential modules $M$ and $M'$,
the direct product $M\oplus M'$ is naturally defined as
$\mathcal{K}$-modules equipped with the derivation $$\nabla_{M\oplus M'}(m+m'):=\nabla_{M}m+\nabla_{M'}m'\quad(m\in M,\,n'\in M').$$
Also we can define the tensor product $M\otimes_{\mathcal{K}} M'$ with the derivation
\[
\nabla_{M\otimes_{\mathcal{K}} M'}(m\otimes m'):=\nabla_{M}m\otimes m'+m\otimes \nabla_{M'}m'
\quad(m\in M,\,n'\in M').
\]
The dual module of $M$ is $M^{*}:=\mathrm{Hom}_{\mathcal{K}}(M,\mathcal{K})$
with the derivation $\nabla_{M^{*}}$ satisfying the following.
If $G$ is the matrix of $\nabla_{M}$ with respect to a basis $\mathbf{e}$,
then $-{}^{t}G$ is the matrix of $\nabla_{M^{*}}$ with respect to the
dual basis $\mathbf{f}$ of $\mathbf{e}$.

The identification $$\mathrm{Hom}_{\mathcal{K}}(M,M')\cong
M^{*}\otimes_{\mathcal{K}} M$$ induces the differential module structure on $\mathrm{Hom}_{\mathcal{K}}(M,M').$

\subsection{Hukuhara-Turrittin-Levelt normal forms}
We shall review the Hukuhara-Turrittin-Levelt theory which gives a formal classification of local differential equations.
We use the notation
\[
\mathrm{diag}(A_{1},A_{2},\ldots,A_{k})
\]
which stands for a block diagonal matrix with the diagonal entries
$A_{i}\in M(n_{i},\mathcal{K})$.
Recall that the substitution $\xi\colon f(z)\mapsto f(e^{2\pi i}z)$ for $f(z)\in \mathbb{C}(\!(z^{\frac{1}
s})\!)$ generates the Galois group
\[
	\mathrm{Gal}(\mathbb{C}(\!(z^{\frac{1}{s}})\!)/\mathbb{C}(\!(z)\!))
	\cong \mu_{s}
\]
where $\mu_{s}$ is the cyclic group which consists of $s$th roots of 1 in
$\mathbb{C}$.

\begin{df}[HTL cell]\normalfont
	Take $q(z)\in \mathcal{P}^{-}$ and set $r=\mathrm{min}\{s\in
	\mathbb{Z}_{>0}
	\mid q(z)\in z^{-\frac{1}{s}}\mathbb{C}[z^{-\frac{1}{s}}]\}$.
	Then the \emph{elementary Hukuhara-Turrittin-Levelt cell}
	$E_{q(z),R}$ for the above $q(z)$ and $R\in M(n,\mathbb{C})$
	is
	\begin{multline*}
		E_{q(z),R}:=\\
		\mathrm{diag}\left(
	q(z)I_{n}+R,\xi(q)(z)I_{n}+R,\ldots,
	\xi^{r-1}(q)(z)I_{n}+R
	\right)z^{-1}\\\in M(rn,\mathbb{C}(\!(z^{\frac{1}{r}})\!))
\end{multline*}
	Here we call the integers $n$ and $r$  \emph{multiplicity} and \emph{ramification index} of
	$E_{q(z),R}$ respectively.
\end{df}

\begin{df}[HTL normal form]\normalfont
A \emph{Hukuhara-Turrittin-Levelt normal form} is a matrix
\[
\mathrm{diag}(E_{q_{1}(z),R_{1}},\ldots,E_{q_{m}(z),R_{m}})
\]
with elementary HTL cells $E_{q_{i}(z),R}$ for $i=1,2,\ldots,m$
such that
\[
\mathrm{Gal}(\mathcal{P}/\mathbb{C}(\!(z)\!))\cdot q_{i}(z)\cap
\mathrm{Gal}(\mathcal{P}/\mathbb{C}(\!(z)\!))\cdot q_{j}(z)=\emptyset,
\quad\quad \text{if }i\neq j.
\]
\end{df}

\begin{thm}[Hukuhara-Turrittin-Levelt, \cite{Huk}, \cite{Tur}, \cite{Lev}]\label{HTL}
Let $M$ be a differential module over $\mathbb{C}(\!\{z\}\!)$ of rank $n$
and $\widetilde{M}:=\mathbb{C}(\!(z)\!)\otimes_{\mathbb{C}(\!\{z\}\!)}M$
the formalization of $M$.
Then there exists an HTL normal form
\[
\mathrm{diag}(E_{q_{1}(z),R_{1}},\ldots,E_{q_{m}(z),R_{m}})
\]
as a matrix of $\overline{M}:=\mathcal{P}\otimes_{\mathbb{C}(\!(z)\!)}\widetilde{M}$
with respect to a suitable basis.
Furthermore, if two differential modules  $M$ and $M'$ over
$\mathbb{C}(\!\{z\}\!)$ share a same HTL normal form, then
$\widetilde{M}\cong \widetilde{M'}$.
\end{thm}

The HTL normal form induces the following decomposition of $M$.
\begin{thm}[see $(7.15)$ in \cite{BV} and COROLLARY 3.3 in \cite{Sab}]\label{refinedHTL}
	We use the same notation as in Theorem \ref{HTL}. There exists
	a differential module $M_{E_{q_{i}(z),R_{i}}}$ over
	$\mathbb{C}(\!\{z\}\!)$ whose HTL normal form is $E_{q_{i}(z),R_{i}}$
	for each $i=1,2,\ldots,m$ and we have a decomposition
	\[
		\widetilde{M}\cong \bigoplus_{i=1}^{m}\widetilde{M_{E_{q_{i}(z),R_{i}}}}
	\]
	as differential modules over $\mathbb{C}(\!(z)\!).$
\end{thm}

\subsection{Komatsu-Malgrange irregularity}
Let us recall that the \emph{index} of a $\mathbb{C}$-linear endmorphism $\Phi$
is
\[
\chi(\Phi):=\mathrm{dim}_{\mathbb{C}}\mathrm{Ker\,}\Phi-
\mathrm{dim}_{\mathbb{C}}\mathrm{Coker\,}\Phi.
\]
The Komatsu-Malgrange irregularity is an analytic invariant
of local differential equations defined as follows.

\begin{df}[Komatsu-Malgrange irregularity]\normalfont
	Let $M$ be a finite differential module over $\mathbb{C}(\!\{z\}\!)$
	and  $\widetilde{M}:=\mathbb{C}(\!(z)\!)\otimes_{\mathbb{C}(\!\{z\}\!)}M$
	its formalization.

	Then the \emph{Komatsu-Malgrange irregularity} of $M$ is
	\[
	\mathrm{Irr}(M):=\chi(\nabla_{\widetilde{M}})-\chi(\nabla_{M}).
	\]

\end{df}

If $M$ has the HTL-normal form
\[
	\mathrm{diag}(E_{q_{1}(z),R_{1}},\ldots,E_{q_{m}(z),R_{m}}),
\]

then it is known that the Komatsu-Malgrange irregularity is
\[
\mathrm{Irr}(M)=-\sum_{i=1}^{m}\sum_{j=0}^{r_{i}-1}\mathrm{ord\,}\xi^{j}(q_{i})(z)=-\sum_{i=1}^{m}r_{i}\,\mathrm{ord\,}q_{i}(z).
\]
Here $r_{i}$ are ramification indices of $E_{q_{i}(z),R_{i}}$ for $i=1,2,\ldots,m$.
\section{Local comparison: Milnor formula}
In this section we deal with a local differential module and define its
characteristic polynomial with respect to a fixed basis.
The zero locus of this characteristic polynomial may have a singularity at
infinity which corresponds to the irregular singularity of the differential module.
We shall compare these singularities and obtain a comparison formula
between the irregularity of differential module and the Milnor number of
the characteristic polynomial.
\subsection{Hukuhara-Turrittin-Levelt normal form and decomposition of characteristic polynomial}\label{HTLMILNOR}

\begin{df}[characteristic polynomial]\normalfont
Let us consider a finite differential module $M$ over $\mathbb{C}(\!\{z\}\!)$ of rank $n$.
Fix a matrix $G\in M(n,\mathbb{C}(\!\{z\}\!))$ of $M$ with respect to a basis
$\mathbf{e}$.
Then the \emph{characteristic polynomial of M  with respect to $\mathbf{e}$}
is
\[
\mathrm{det}(z^{\nu}yI_{n}-z^{\nu}G)=z^{n\nu}\mathrm{det}(yI_{n}-G)\in \mathbb{C}\{z\}[y].
\]
Here $\nu\in \mathbb{Z}_{\ge 0}$ is the pole order of $G$.
\end{df}

The characteristic polynomial may have a singularity at $(y,z)=(\infty,0)$.
We shall see that the HTL normal form of $M$ have some information on the
singularity.

\begin{df}[multiplicity free HTL normal form]\label{mfhtl}\normalfont
	An HTL normal form
	\[
		\mathrm{diag}(E_{q_{1}(z),R_{1}},\ldots,E_{q_{m}(z),R_{m}})
	\]
	is said to be \emph{multiplicity free} when all HTL cells $E_{q_{i}(z),R_{i}}$, $i=1,2,\ldots,m$, are multiplicity one,
	namely,
	$R_{i}\in M(1,\mathbb{C})$ for all $i=1,2,\ldots,m$.
\end{df}

\begin{prop}\label{Hensel}
	Let $M$ be a differential module over $\mathbb{C}(\!\{z\}\!)$ of rank $n$
	and fix a matrix $G\in M(n,\mathbb{C}(\!\{z\}\!))$ of $M$.

	Suppose that the $M$ has the multiplicity free HTL normal form
	 \[
		\mathrm{diag}(E_{q_{1}(z),R_{1}},\ldots,E_{q_{m}(z),R_{m}}).
	 \]
	 Then the characteristic polynomial $z^{n\nu}\mathrm{det}\left(yI_{n}-G\right)\in \mathbb{C}\{z\}[y]$
	 decomposes as follows,
	 \[
		z^{n\nu}\mathrm{det}\left(yI_{n}-G\right)=z^{n\nu}\prod_{i=1}^{m}\prod_{j=0}^{r_{i}-1}
		\left(y-\frac{\tilde{q}_{[i,j]}(z)}{z}\right).
	 \]
	 Here $\tilde{q}_{[i,j]}(z)\in \mathcal{P}^{\text{conv}}$ satisfies that
	 \[
		\pr^{-}(\tilde{q}_{[i,j]}(z))= \xi^{j}q_{i}(z)
	 \]
	 for each $i=1,2,\ldots,m,\,j=0,1,\ldots,r_{i}-1$,
	 and $\pr^{-}\colon \mathcal{P}\rightarrow \mathcal{P}^{-}$ is
	 the projection along the decomposition $\mathcal{P}=\mathcal{P}^{-}\oplus \mathcal{P}^{+}.$
\end{prop}
\begin{proof}

	For $q_{i}(z)$, $i=1,2,\ldots,m$,  define
	\[
		E^{o}_{q_{i}(z)}:=\mathrm{diag}\left(q_{i}(z),\xi(q_{i})(z),
		\ldots,\xi^{r_{i}-1}(q_{i})(z)\right)z^{-1}\in
		M(r_{i},\mathbb{C}(\!(z^{\frac{1}{r_{i}}})\!)).
	\]
	Then the multiplicity free condition leads to
	\[
		\mathrm{diag}(E_{q_{1}(z),R_{1}},\ldots,E_{q_{m}(z),R_{m}})\equiv
		\mathrm{diag}(E^{o}_{q_{1}(z)},\ldots,E^{o}_{q_{m}(z)})
		\quad(\mathrm{mod }z^{-1}\mathbb{C}[\![z^{\frac{1}{s}}]\!]).
	\]
	Here $s:=\mathrm{lcm}\{r_{1},r_{2},\ldots,r_{m}\}$.
	Thus there exists  $X\in \mathrm{GL}(n,\mathbb{C}(\!(z^{\frac{1}{s}})\!))$
	such that
	\begin{align*}
	XGX^{-1}+\left(\frac{d}{dz}X\right)X^{-1}\equiv 	\mathrm{diag}(E^{o}_{q_{1}(z)},\ldots,E^{o}_{q_{m}(z)})
	\quad(\mathrm{mod }z^{-1}\mathbb{C}[\![z^{\frac{1}{s}}]\!]).
	\end{align*}
	Since $\mathrm{ord\ }\left(\frac{d}{dz}X\right)X^{-1}\ge -1$,
	we have
	\[
	XGX^{-1}\equiv  	\mathrm{diag}(E^{o}_{q_{1}(z)},\ldots,E^{o}_{q_{m}(z)})
	\quad(\mathrm{mod }z^{-1}\mathbb{C}[\![z^{\frac{1}{s}}]\!]).
	\]
	Note that all the entries $\xi^{i}(q_{j})(z)$ in $\mathrm{diag}(E^{o}_{q_{1}(z)},\ldots,E^{o}_{q_{m}(z)})$ are mutually different.
	Thus applying the Lemma \ref{splitting} below repeatedly,
	 we can find $X'\in \mathrm{GL}(n,\mathbb{C}(\!(z^{\frac{1}{s}})\!))$
	so that $X'XG(X'X)^{-1}$ is a diagonal matrix and
	\[
		X'XG(X'X)^{-1}\equiv 	\mathrm{diag}(E^{o}_{q_{1}(z)},\ldots,E^{o}_{q_{m}(z)})
	\quad(\mathrm{mod }z^{-1}\mathbb{C}[\![z^{\frac{1}{s}}]\!]).
	\]
	This leads us to the decomposition
	\begin{equation}\label{decomp}
	 \mathrm{det}\left(yI_{n}-G\right)=\prod_{i=1}^{m}\prod_{j=0}^{r_{i}-1}
	 \left(y-\frac{\tilde{q}_{[i,j]}(z)}{z}\right)
 \end{equation}
	with $\tilde{q}_{[i,j]}(z)\in \mathbb{C}(\!(z^{\frac{1}{s}})\!)$ satisfying
	\begin{equation*}
	 \tilde{q}_{[i,j]}(z)\equiv 
	 \xi^{j}(q_{i})(z)\quad(\mathrm{mod\ } \mathbb{C}[\![z^{\frac{1}{s}}]\!])
 \end{equation*}
	for each $i=1,2,\ldots,m,\,j=0,1,\ldots,r_{i}-1.$
	Since the field $\mathcal{P}^{\text{conv}}$
	is algebraically closed, the equation (\ref{decomp}) coincides with
	the decomposition in $\mathcal{P}^{\text{conv}}[y]$. Thus
	the formal Puiseux series $\tilde{q}_{[i,j]}(z)$ should be convergent power series.
\end{proof}
The following lemma is just a slight modification of the standard and well-known argument in the local formal theory of differential equations, so called the splitting lemma, see Lemma 3 in the section
3.2 in \cite{Bal} for example.
\begin{lem}\label{splitting}
	Let us consider $A(t)=t^{r}\sum_{i=0}^{\infty}A_{i}t^{i}\in
	M(n,\mathbb{C}(\!(t)\!))$
	and suppose that
	$A_{0}=\mathrm{diag}(A^{11}_{0},A^{22}_{0})\in M(n,\mathbb{C})$
	with $A^{jj}_{0}\in M(n_{j},\mathbb{C})$, $j=1,2$, and
	the sets of eigenvalues of
	$A_{0}^{11}$ and $A_{0}^{22}$ respectively are disjoint.
	Then there exists
	\[
	T(z)=\begin{pmatrix}
		I_{n_{1}}&T_{12}(t)\\
		T_{21}(z)&I_{n_{2}}
		\end{pmatrix},\quad
	T_{jk}(t)=\sum_{i=1}^{\infty}T^{jk}_{i}t^{i},\ j,k=1,2
	\]
	such that
	\[
	T(t)A(t)T^{-1}(t)=\begin{pmatrix}
		B_{11}(t)&0\\
		0&B_{22}(t)
	\end{pmatrix}
	\]
	where $B_{jj}(t)=t^{r}\sum_{i=0}^{\infty}B^{jj}_{i}t^{i}\in
	M(n_{j},\mathbb{C}(\!(t)\!))$ with $B_{0}^{jj}=A_{0}^{jj}$,
	 $j=1,2.$
\end{lem}
\begin{proof}
	The proof is almost the same as that of the splitting lemma
	in the local theory of differential equations.

	Let us write
	\[
	A(t)=\begin{pmatrix}
		A_{11}(t)&A_{12}(t)\\
		A_{21}(t)&A_{22}(t)
	\end{pmatrix},
	\]
	where $A_{jk}(t)=t^{r}\sum_{i=0}^{\infty}A^{jk}_{i}t^{i}\in
	M(n_{j}\times n_{k},\mathbb{C}(\!(t)\!)),
	\ j,k=1,2.$
	Then the  equation
	\[
		T(t)A(t)=\begin{pmatrix}
			B_{11}(t)&0\\
			0&B_{22}(t)
		\end{pmatrix}	T(t)
	\]
	is equivalent to
	\begin{align*}
	&B_{jj}(t)=A_{jj}(t)+T_{jk}(t)A_{kj}(t)\\
	&A_{jk}(t)+T_{jk}(t)A_{kk}(t)=B_{jj}(t)T_{jk}(t)
	\end{align*}
	for $1\le j\neq k\le 2.$
	Comparing the coefficients of the powers of $t$ on
	both sides, we have
	\begin{align*}
	T_{n}^{jk}A_{0}^{kk}-A_{0}^{jj}T_{n}^{jk}&=
	\sum_{\mu=1}^{n-1}(A^{jj}_{n-\mu}T_{\mu}^{jk}-T_{\mu}^{jk}A_{n-\mu}^{kk})
	\\&\quad -\sum_{\nu=1}^{n-2}T_{\nu}^{jk}\sum_{\mu=1}^{n-\nu-1}A_{n-\mu-\nu}^{kj}T_{\mu}^{jk}+A_{n}^{jk}
	\end{align*}
	for $n\ge 1.$
	Recall that the equation
	\[
	TA_{0}^{22}-A_{0}^{11}T=C
	\]
	for a given $C\in M(n_{1}\times n_{2},\mathbb{C})$ has the unique
	solution $T\in M(n_{1}\times n_{2},\mathbb{C})$ since
	the sets of eigenvalues of
	$A_{0}^{11}$ and $A_{0}^{22}$ respectively are disjoint,
	see Lemma 24 of the section A.1 in \cite{Bal} for example.
	Thus the above equations determine $T_{n}^{jk}$, $n=1,2,\ldots,$
	inductively.
\end{proof}

Since the HTL normal form is multiplicity free,
the decomposition in Theorem \ref{refinedHTL}
\[
		\widetilde{M}\cong \bigoplus_{i=1}^{m}\widetilde{M_{E_{q_{i}(z),R_{i}}}}
\]
is the irreducible decomposition.
Correspondingly, the following proposition shows that
the decomposition in Proposition \ref{Hensel} is the irreducible
decomposition with
the irreducible components
\[
 \prod_{j=0}^{r_{i}-1}
 \left(y-\frac{\tilde{q}_{[i,j]}(z)}{z}\right)\in \mathbb{C}(\!\{z\}\!)[y],\quad i=1,2,\ldots,m.
\]

\begin{prop}\label{Galois}
	We use the same notation as in Proposition \ref{Hensel}.
	The Galois orbit of $\tilde{q}_{[i,j]}(z)\in \mathcal{P}^{\text{conv}}$
	is
	\begin{align*}
	\mathrm{Gal}(\mathcal{P}/\mathbb{C}(\!(z)\!))\cdot \tilde{q}_{[i,j]}(z)&=
	\{\tilde{q}_{[i,0]}(z),\tilde{q}_{[i,1]}(z),\ldots,\tilde{q}_{[i,r_{i}-1]}(z)\}\\
	&=\{\tilde{q}_{[i,0]}(z),\xi(\tilde{q}_{[i,0]})(z),\ldots,\xi^{r_{i}-1}(\tilde{q}_{[i,0]})(z)\}
\end{align*}
	for each $i=1,2,\ldots,m,\,j=0,1,\ldots,r_{i-1}.$
	In particular $\tilde{q}_{[i,j]}(z)\in \mathbb{C}(\!\{z^{\frac{1}{r_{i}}}\}\!).$
\end{prop}
\begin{proof}
	The decomposition
	\[
	 \mathrm{det}\left(yI_{n}-G\right)=\prod_{i=1}^{m}\prod_{j=0}^{r_{i}-1}
	 \left(y-\frac{\tilde{q}_{[i,j]}(z)}{z}\right)\in \mathbb{C}(\!\{z\}\!)[y]
	\]
	tells us that
	\[
	\mathrm{Gal}(\mathcal{P}/\mathbb{C}(\!(z)\!))\cdot \tilde{q}_{[i,j]}(z)\subset
	\{\tilde{q}_{[k,l]}(z)\mid k=1,2,\ldots,m,\,l=0,1,\ldots,r_{k}\}.
	\]
	Let $\pr^{-}\colon \mathcal{P}\rightarrow \mathcal{P}^{-}$ be the projection along the decomposition $\mathcal{P}=\mathcal{P}^{-}\oplus
	\mathcal{P}^{+}$.
	Since $\pr^{-}$ is compatible with the Galois action,
	we have
	\begin{align*}
	\pr^{-}\left(\mathrm{Gal}(\mathcal{P}/\mathbb{C}(\!(z)\!))\cdot \tilde{q}_{[i,j]}(z)\right)&=\mathrm{Gal}(\mathcal{P}/\mathbb{C}(\!(z)\!))\cdot \xi^{j}(q_{i})(z)\\
	&=\mathrm{Gal}(\mathcal{P}/\mathbb{C}(\!(z)\!))\cdot q_{i}(z).
\end{align*}
The restriction
\[
\widetilde{\pr}^{-}:=\pr^{-}|_{\{\tilde{q}_{[k,l]}(z)\mid k=1,2,\ldots,m,\,l=0,1,\ldots,r_{k}\}}
\]
is a bijection onto $\{\xi^{l}(q_{k})(z)\mid k=1,2,\ldots,m,\,l=0,1,\ldots,r_{k}\}$.
Thus we have
	\begin{align*}
	\mathrm{Gal}(\mathcal{P}/\mathbb{C}(\!(z)\!))\cdot \tilde{q}_{[i,j]}(z)&=
	(\widetilde{pr}^{-})^{-1}\left(\mathrm{Gal}(\mathcal{P}/\mathbb{C}(\!(z)\!))\cdot q_{i}(z)\right)\\
	&=
	\{\tilde{q}_{[i,0]}(z),\tilde{q}_{[i,1]}(z),\ldots,\tilde{q}_{[i,r_{i}-1]}(z)\}\\
	&=\{\tilde{q}_{[i,0]}(z),\xi(\tilde{q}_{[i,0]})(z),\ldots,\xi^{r_{i}-1}(\tilde{q}_{[i,0]})(z)\}.
\end{align*}
\end{proof}
\subsection{Milnor formula}

By Proposition \ref{Galois} , the decomposition in Proposition \ref{Hensel} can be rewritten as follows,
\begin{align*}
	z^{n\nu}\mathrm{det}\left(yI_{n}-G\right)&=z^{n\nu}\prod_{i=1}^{m}\prod_{j=0}^{r_{i}-1}
 \left(y-\frac{\xi^{j}(\tilde{q}_{i})(z)}{z}\right)\\
 &=\prod_{i=1}^{m}\prod_{j=0}^{r_{i}-1}z^{\nu_{i}}
 \left(y-\frac{\xi^{j}(\tilde{q}_{i})(z)}{z}\right),
\end{align*}
where $\tilde{q}_{i}(z)\in \mathbb{C}(\!\{z^{\frac{1}{r_{i}}}\}\!)$ satisfy
$\pr^{-}(\tilde{q}_{i}(z))=q_{i}(z)$
and we set 
$
	\nu_{i}:=r_{i}\cdot\mathrm{max}\left\{0,-\mathrm{ord\,}\frac{q_{i}(z)}{z}\right\}	
$
for $i=1,2,\ldots,m.$
Moreover this is the irreducible decomposition with the irreducible components
\[
z^{\nu_{i}}\prod_{j=0}^{r_{i}-1}
\left(y-\frac{\xi^{j}(\tilde{q}_{i})(z)}{z}\right)\in \mathbb{C}\{z\}[y].
\]

We now investigate the singularity of the zero locus of each irreducible components
at $(y,z)=(\infty,0)$.
To be more precise, let us put $y=\frac{\eta}{\zeta}$
and consider the homogenized polynomial
\[
z^{\nu_{i}}\prod_{j=0}^{r_{i}-1}
\left(\eta-\frac{\xi^{j}(\tilde{q}_{i})(z)}{z}\zeta\right).
\]
The restriction to $\eta=1$ gives
\[
z^{\nu_{i}}\prod_{j=0}^{r_{i}-1}
\left(1-\frac{\xi^{j}(\tilde{q}_{i})(z)}{z}\zeta\right)=
z^{\nu_{i}}\prod_{j=0}^{r_{i}-1}\left(-\frac{\xi^{j}(\tilde{q}_{i})(z)}{z}\right)\cdot
\prod_{j=0}^{r_{i}-1}
\left(\zeta-\frac{z}{\xi^{j}(\tilde{q}_{i})(z)}\right).
\]
Suppose that $\mathrm{ord\,}\frac{\tilde{q_{i}}(z)}{z}<0$.
Then  $$z^{\nu_{i}}\prod_{j=0}^{r_{i}-1}\left(-\frac{\xi^{j}(\tilde{q}_{i})(z)}{z}\right)|_{z=0}\neq 0\quad
\text{and}\quad  \frac{z}{\xi^{j}(\tilde{q}_{i})(z)}\in \mathbb{C}(\!\{z^{\frac{1}{r_{i}}}\}\!)\cap
\mathbb{C}[\![z^{\frac{1}{r_{i}}}]\!],$$
and the  zero locus of
\[
\prod_{j=0}^{r_{i}-1}
\left(\zeta-\frac{z}{\xi^{j}(\tilde{q}_{i})(z)}\right)
\]
defines the plane curve germ $C_{\tilde{q_{i}}}$  at $(\zeta,z)=(0,0)$ where $\zeta:=\frac{1}{y}$.

Let us set 
\[
	-\frac{p_{i}}{r_{i}}:=\mathrm{ord\ }q_{i}(z)	 
\] 
with relatively prime integers $p_{i}$ and $r_{i}$ for $i=1,2,\ldots,m.$
\begin{prop}\label{intersection}
	Let us fix $i\neq j\in\{1,2,\ldots,m\}$ such that 
	$\mathrm{ord\,}\frac{\tilde{q_{i}}(z)}{z}<0$ and $\mathrm{ord\,}\frac{\tilde{q_{j}}(z)}{z}<0$.
	Then the intersection number of $C_{\tilde{q_{i}}}$ and
	$C_{\tilde{q_{j}}}$ is
	\[
	(C_{\tilde{q_{i}}},C_{\tilde{q_{j}}})=
	p_{i}r_{j}+p_{j}r_{i}+r_{i}r_{j}-
	\mathrm{Irr}(\mathrm{Hom}_{\mathbb{C}(\!\{z\}\!)}(M_{E_{q_{i},R_{i}}},M_{E_{q_{j},R_{j}}})).
	\]
\end{prop}
\begin{proof}
Since the plane curve germ $C_{\tilde{q_{i}}}$ is parametrized by
$z(t)=t^{r_{i}}, \zeta(t)=\frac{t^{r_{i}}}{\tilde{q_{i}}(t^{r_{i}})}$ and
the germ $C_{\tilde{q_{j}}}$ is defined by $\prod_{k=1}^{r_{j}}\left(
	\zeta-\frac{z}{\xi^{k}(\tilde{q_{j}})(z)}
\right)=0$.
Then the  intersection number of them is computed as follows,
see 1.2 in \cite{Wal} for example.
\begin{align*}
	(C_{\tilde{q_{i}}},C_{\tilde{q_{j}}})&=\mathrm{ord\,}_{t}
	\prod_{k=1}^{r_{j}}
		\left(\zeta(t)-\frac{t^{r_{i}}}{\xi^{k}(\tilde{q_{j}})(t^{r_{i}})}\right)\\
		&=r_{i}\mathrm{ord\,}_{z}\prod_{k=1}^{r_{j}}
		\left(\frac{z}{\tilde{q_{i}}(z)}-\frac{z}{\xi^{k}(\tilde{q_{j}})(z)}\right)\\
		&=r_{i}\mathrm{ord\,}_{z}\prod_{k=1}^{r_{j}}
		\left(\frac{1}{\tilde{q_{i}}(z)}-\frac{1}{\xi^{k}(\tilde{q_{j}})(z)}\right)+r_{i}r_{j}\\
		&=r_{i}\mathrm{ord\,}_{z}\prod_{k=1}^{r_{j}}
		\left(\frac{\xi^{k}(\tilde{q_{j}})(z)-\tilde{q_{i}}(z)}{\tilde{q_{i}}(z)\xi^{k}(\tilde{q_{j}})(z)}\right)+r_{i}r_{j}\\
		&=r_{i}\mathrm{ord\,}_{z}\prod_{k=1}^{r_{j}}
		\left(\xi^{k}(\tilde{q_{j}})(z)-\tilde{q_{i}}(z)\right)+p_{i}r_{j}+p_{j}r_{i}+r_{i}r_{j}\\
		&=p_{i}r_{j}+p_{j}r_{i}+r_{i}r_{j}-
		\mathrm{Irr}(\mathrm{Hom}_{\mathbb{C}(\!\{z\}\!)}(M_{E_{q_{i},R_{i}}},M_{E_{q_{j},R_{j}}})).
\end{align*}
\end{proof}
\begin{prop}\label{milnor}
	Let us fix an $i\in \{1,2,\ldots,m\}$ and 
	suppose that $\mathrm{ord\,}\frac{\tilde{q_{i}}(z)}{z}<0$.
	Then the Milnor number of the germ $C_{\tilde{q}_{i}}$  is
	\[
	\mu(C_{\tilde{q_{i}}})=(2p_{i}+r_{i}-1)(r_{i}-1)
	-\mathrm{Irr}\left(\mathrm{End}_{\mathbb{C}(\!\{z\}\!)}(M_{E_{q_{i},R_{i}}})
	\right).
	\]
\end{prop}
\begin{proof}
	Set $F_{i}(\zeta,z):=\prod_{k=1}^{r_{i}}\left(
		\zeta-\frac{z}{\xi^{k}(\tilde{q_{i}})(z)}
	\right)$. Then the germ $C_{\tilde{q_{i}}}$ is defined by $F_{i}=0$.
	Then the Milnor number can be obtained by
	\[
		\mu(C_{\tilde{q_{i}}})=(F_{i},\frac{\partial}{\partial \zeta}F_{i})+1
		-(F_{i},z).
	\]
	We refer to COROLLARY 7.16 and THEOREM 7.18 in \cite{Hef} for this fact.

	If we note that $\displaystyle \frac{\partial}{\partial\zeta}F_{i}=\sum_{k=1}^{r_{i}}
	\prod_{\substack{1\le l\le r_{i}\\l\neq k}}(\zeta-\frac{z}{\xi^{l}(\tilde{q_{i}})(z)}),$
	then
	\begin{align*}
		&(F_{i},\frac{\partial}{\partial \zeta}F_{i})=\mathrm{ord\,}_{z}\prod_{\substack{1\le l,k\le r_{i}\\l\neq k}}
		\left(\frac{z}{\xi^{k}(\tilde{q_{i}})(z)}-\frac{z}{\xi^{l}(\tilde{q_{i}})(z)}\right)\\
		&\quad\quad =\mathrm{ord\,}_{z}\prod_{\substack{1\le l,k\le r_{i}\\l\neq k}}
		\left(\xi^{l}(\tilde{q_{i}})(z)-\xi^{k}(\tilde{q_{i}})(z)\right)+r_{i}(r_{i}-1)+2p_{i}(r_{i}-1)\\
		&\quad\quad =-\mathrm{Irr}\left(\mathrm{End}_{\mathbb{C}(\!\{z\}\!)}(M_{E_{q_{i},R_{i}}})
		\right)+r_{i}(r_{i}-1)+2p_{i}(r_{i}-1).
	\end{align*}
	Also we have $(F_{j},z)=r_{i}$.
	Thus combining these equations, we have
	\begin{align*}
	\mu(C_{\tilde{q_{i}}})&=(2p_{i}+r_{i})(r_{i}-1)-\mathrm{Irr}\left(\mathrm{End}_{\mathbb{C}(\!\{z\}\!)}(M_{E_{q_{i},R_{i}}})
	\right)+1-r_{i}\\
	&=(2p_{i}+r_{i}-1)(r_{i}-1)-\mathrm{Irr}\left(\mathrm{End}_{\mathbb{C}(\!\{z\}\!)}(M_{E_{q_{i},R_{i}}})
	\right).
	\end{align*}
\end{proof}

Now we compute the Milnor number of the zero locus of the characteristic polynomial
$z^{n\nu}\mathrm{det}(yI_{n}-G)$ at $(y,z)=(\infty,0)$ as follows.
Let us
suppose that $M$ has a singularity at $z=0$.
Then $G$ has a pole at $z=0$ and the zero locus of the homogenization
of the characteristic polynomial
\[
	z^{n\nu}\zeta^{n}\mathrm{det}\left(\frac{\eta}{\zeta}I_{n}-G\right)
\]
pass through the point $([\zeta:\eta],z)=([0:1],0)$.
Let us denote the zero locus  by $C_{G}$ and  the Milnor number of $C_{G}$ at $([\zeta:\eta],z)=([0:1],0)$
by $\mu(C_{G})_{\infty}$.

\begin{thm}[Milnor formula]\label{premilnor}
	Let us take a differential module $M$ over $\mathbb{C}(\!\{z\}\!)$
	of rank $n$
	and a matrix $G$ of $M$ as in Proposition \ref{Hensel}.
	Suppose that $M$ has a singularity at $z=0$.
	Then the Milnor number of $C_{G}$ at $([\zeta:\eta],z)=([0:1],0)$ is
	\[
		\mu(C_{G})_{\infty}=-n^{2}-\mathrm{Irr}(\mathrm{End}_{\mathbb{C}(\!\{z\}\!)}(M))+2(n-1)(C_{G},\zeta)_{\infty}+(m-r_{C_{G}})+1.
	\]
	Here $r_{C_{G}}$ is the number of branches of the germ of $C_{G}$ at $([\zeta:\eta],z)=([0:1],0)$ and
	$(C_{G},\zeta)_{\infty}$ is the intersection number of $C_{G}$ and
	$\zeta$ at $([\zeta:\eta],z)=([0:1],0)$.
\end{thm}
\begin{proof}
	First we assume that  $\mathrm{ord\,}\frac{\tilde{q_{i}}(z)}{z}<0$ for all $i=1,2,\ldots,m$.
	Then the decomposition in Proposition \ref{Hensel} shows that the homogenized
	characteristic polynomial
	\begin{align*}
		z^{n\nu}\zeta^{n}\mathrm{det}\left(\frac{\eta}{\zeta}I_{n}-G\right)&=z^{n\nu}\prod_{i=1}^{m}\left(
			\prod_{j=0}^{r_{i}-1}\left(-\frac{\xi^{j}(\tilde{q}_{i})(z)}{z}\right)\cdot
	\prod_{j=0}^{r_{i}-1}
	\left(\zeta-\frac{z}{\xi^{j}(\tilde{q}_{i})(z)}\eta\right)
		\right)
	\end{align*}
	defines a reduced plane curve germ $C_{G}$ at $([\zeta:\eta],z)=([0:1],0)$
	with branches $C_{\tilde{q}_{i}}$, $i=1,2,\ldots,m$.
	Then by Propositions \ref{intersection} and \ref{milnor},
	the Milnor number of $C_{G}$ is
	\begin{align*}
		\mu(C_{G})_{\infty}&=\sum_{i=1}^{m}\mu(C_{\tilde{q_{i}}})+2\sum_{1\le j<k\le m}(C_{\tilde{q_{j}}},C_{\tilde{q_{k}}})
		-r_{C_{G}}+1\\
		&=-\sum_{1\le i,j\le m}\mathrm{Irr}(\mathrm{Hom}_{\mathbb{C}(\!\{z\}\!)}(M_{E_{q_{i}(z),R_{i}}},M_{E_{q_{j}(z),R_{j}}}))\\
		&\quad+\sum_{i=1}^{m}(2p_{i}+r_{i}-1)(r_{i}-1)
		+2\sum_{1\le j<k\le m}(p_{j}r_{k}+p_{k}r_{j}
	+r_{j}r_{k})\\	&\quad -m+1\\
		&=-\mathrm{Irr}(\mathrm{End}_{\mathbb{C}(\!\{z\}\!)}(M))+\sum_{i=1}^{m}(2p_{i}r_{i}+r_{i}r_{i}-2(p_{i}+r_{i}))+m\\
		&\quad +2\sum_{1\le j<k\le m}(p_{j}r_{k}+p_{k}r_{j}+r_{j}r_{k})-m+1\\
	\end{align*}
	\begin{align*}
		&=-\mathrm{Irr}(\mathrm{End}_{\mathbb{C}(\!\{z\}\!)}(M))+2\sum_{i=1}^{m}p_{i}\sum_{i=1}^{m}r_{i}
		+\sum_{i=1}^{m}r_{i}\sum_{i=1}^{m}r_{i}\\
		&\quad -2\sum_{i=1}^{m}(p_{i}+r_{i})+1.
	\end{align*}
	Now let us note that $n=\sum_{i=1}^{m}r_{i}$ and
	\begin{align*}
		(C_{G},\zeta)&=\sum_{i=1}^{m}(F_{i},\zeta)
		=\sum_{i=1}^{m}\mathrm{ord\ }_{t}\frac{t^{r_{i}}}{\tilde{q_{i}}(t^{r_{i}})}\\
		&=\sum_{i=1}^{m}(p_{i}+r_{i}).
	\end{align*}
	Then we have
	\begin{align*}
		\mu(C_{G})_{\infty}&=-\mathrm{Irr}(\mathrm{End}_{\mathbb{C}(\!\{z\}\!)}(M))+2\sum_{i=1}^{m}(p_{i}+r_{i}-r_{i})n+n^{2}\\
		&\quad -2\sum_{i=1}^{m}(p_{i}+r_{i})+1\\
		&=-\mathrm{Irr}(\mathrm{End}_{\mathbb{C}(\!\{z\}\!)}(M))+2(n-1)\sum_{i=1}^{m}(p_{i}+r_{i})-n^{2}+1\\
		&=-n^{2}-\mathrm{Irr}(\mathrm{End}_{\mathbb{C}(\!\{z\}\!)}(M))+2(n-1)(C_{G},\zeta)+1.
	\end{align*}

	On the other hand, let us assume that  there exists $i\in \{1,2,\ldots,m\}$
	such that 
	$\mathrm{ord\,}\frac{\tilde{q}_{i}(z)}{z}\ge 0$. Then $\pr^{-}(\tilde{q}_{i}(z))=q_{i}(z)$ must be $0$ and 
	$\mathrm{ord\,}\frac{\tilde{q_{j}}(z)}{z}<0$ for the other $j\in \{1,2,\ldots,m\}\backslash\{i\}$ 
	because
	$q_{j}(z)\neq 0$ by the definition of HTL normal forms.
	We may put  $i=m$ by permuting the indices if
	necessary.

	Let us note that $m\ge 2$ in this case.
	If $m=1$, then $\frac{\tilde{q}_{m}(z)}{z}=G$.  Hence $G$ has no
	 pole at $z=0$ and
	$M$ has no singularity at $z=0$.

	In a way similar to the above argument, we can  show that
	\begin{align*}
		\mu(C_{G})_{\infty}&=\sum_{i=1}^{m-1}\mu(C_{\tilde{q_{i}}})+2\sum_{1\le j<k\le m-1}(C_{\tilde{q_{j}}},C_{\tilde{q_{k}}})
		-r_{C_{G}}+1\\
		&=-\sum_{1\le i,j\le m-1}\mathrm{Irr}(\mathrm{Hom}_{\mathbb{C}(\!\{z\}\!)}(M_{E_{q_{i}(z),R_{i}}},M_{E_{q_{j}(z),R_{j}}}))\\
		&\quad+2\sum_{i=1}^{m-1}p_{i}\sum_{i=1}^{m-1}r_{i}
		+\sum_{i=1}^{m-1}r_{i}\sum_{i=1}^{m-1}r_{i}-2\sum_{i=1}^{m-1}(p_{i}+r_{i})+1.
	\end{align*}
	Now let us notice that
	\begin{align*}
	\mathrm{Irr}(\mathrm{End}_{\mathbb{C}(\!\{z\}\!)}(M))&=\sum_{1\le i,j\le m-1}\mathrm{Irr}(\mathrm{Hom}_{\mathbb{C}(\!\{z\}\!)}
	(M_{E_{q_{i}(z),R_{i}}},M_{E_{q_{j}(z),R_{j}}}))\\
	&\quad+
	2\sum_{i=1}^{m-1}\mathrm{Irr}(\mathrm{Hom}_{\mathbb{C}(\!\{z\}\!)}(M_{E_{q_{m}(z),R_{m}}},
	M_{E_{q_{i}(z),R_{i}}}))\\
	&\quad +\mathrm{Irr}(\mathrm{End}_{\mathbb{C}(\!\{z\}\!)}(M_{E_{q_{m}(z),R_{m}}}))\\
	&=\sum_{1\le i,j\le m-1}\mathrm{Irr}(\mathrm{Hom}_{\mathbb{C}(\!\{z\}\!)}(M_{E_{q_{i}(z),R_{i}}},M_{E_{q_{j}(z),R_{j}}}))\\
	&\quad +2\sum_{i=1}^{m-1}p_{i},\\
	\sum_{i=1}^{m}r_{i}\sum_{i=1}^{m}r_{i}&=\sum_{i=1}^{m-1}r_{i}\sum_{i=1}^{m-1}r_{i}+2\sum_{i=1}^{m-1}r_{i}+1\\
	&=\sum_{i=1}^{m-1}r_{i}\sum_{i=1}^{m-1}r_{i}+2(n-1)+1,\\
	\sum_{i=1}^{m}(p_{i}+r_{i})&=\sum_{i=1}^{m-1}(p_{i}+r_{i})+1,\\
		(C_{G},\zeta)_{\infty}&=\sum_{i=1}^{m-1}(p_{i}+r_{i})=\sum_{i=1}^{m}(p_{i}+r_{i})-1,
	\end{align*}
	where we use the fact $p_{m}=0$ and $r_{m}=1$.
	Then it follows that
	\begin{align*}
		\mu(C_{G})_{\infty}
		&=-\mathrm{Irr}(\mathrm{End}_{\mathbb{C}(\!\{z\}\!)}(M)
		+2\sum_{i=1}^{m}p_{i}\sum_{i=1}^{m}r_{i}
		+\sum_{i=1}^{m}r_{i}\sum_{i=1}^{m}r_{i}\\
		&\quad -2\sum_{i=1}^{m}(p_{i}+r_{i})-2(n-1)+2\\
		&=-n^{2}-\mathrm{Irr}(\mathrm{End}_{\mathbb{C}(\!\{z\}\!)}(M))
		+2(n-1)(C_{G},\zeta)_{\infty}+2\\
		&=-n^{2}-\mathrm{Irr}(\mathrm{End}_{\mathbb{C}(\!\{z\}\!)}(M))
		+2(n-1)(C_{G},\zeta)_{\infty}+(m-r_{C_{G}})+1.
	\end{align*}
	\end{proof}
	\subsection{Differential modules with regular semisimple matrices over $\mathbb{C}[\![z]\!]$}
In order to decompose the characteristic polynomial in accordance with
the  HTL normal form of the corresponding differential module, we have assumed the multiplicity free condition in
Proposition \ref{Hensel}.
However if we consider a differential module with regular singularity, this module should be just rank 1 under this condition.
Thus we now discuss another condition which we call regular semisimplicity
 over $\mathbb{C}[\![z]\!]$ and see that the previous argument is also
 valid under this condition.

\begin{df}\label{rssdef}\normalfont
	Let us consider a differential module $M$ over $\mathbb{C}(\!\{z\}\!)$
	of rank $n$ with a matrix $G$.
If there exists $X\in \mathrm{GL}(n,\mathbb{C}[\![z]\!])$ such that
\[
XGX^{-1}+\left(\frac{d}{dz}X\right)X^{-1}=\mathrm{diag}(q_{1}(z),q_{2}(z),\ldots,q_{n}(z))z^{-1}
\]
with mutually different polynomials $q_{i}(z)\in \mathbb{C}[z^{-1}]$ of $z^{-1}$, $i=1,2,\ldots,n$,
then we say that $G$ is \emph{regular semisimple over} $\mathbb{C}[\![z]\!]$
with the HTL normal form $\mathrm{diag}(q_{1}(z),q_{2}(z),\ldots,q_{n}(z))z^{-1}$.
\end{df}
\begin{rem}\normalfont
	Let be $M$ a differential module over $\mathbb{C}(\!\{z\}\!)$
	with a matrix $G$. Then in papers \cite{JMU} and \cite{Boa}, it is assumed that the leading coefficient 
	of $G$ is diagonalizable with distinct eigenvalues if the pole order of $G$ at $z=0$ is greater than 1, or 
	diagonalizable with distinct eigenvalues mod $\mathbb{Z}$ if $G$ has a simple pole at $z=0$, see equation $(1.3)$
	in \cite{JMU} and DEFINITION 2.2 in \cite{Boa}.
	Then we can see that $G$ is regular semisimple under this condition.
\end{rem}
\begin{prop}\label{rss}
	Let $M$ be a differential module over $\mathbb{C}(\!\{z\}\!)$
	of rank $n$
	and fix a matrix $G\in M(n,\mathbb{C}(\!\{z\}\!))$ of $M$.

	Suppose that $G$ is regular semisimple over $\mathbb{C}[\![z]\!]$
	with the HTL normal form
	\[
		\mathrm{diag}(q_{1}(z),q_{2}(z),\ldots,q_{n}(z))z^{-1}.
	\]

	 Then the characteristic polynomial $z^{n\nu}\mathrm{det}\left(yI_{n}-G\right)\in \mathbb{C}\{z\}[y]$
	 decomposes as follows,
	 \[
		z^{n\nu}\mathrm{det}\left(yI_{n}-G\right)=z^{n\nu}\prod_{i=1}^{}
		\left(y-\frac{\tilde{q}_{i}(z)}{z}\right).
	 \]
	 Here $\tilde{q}_{i}(z)\in \mathbb{C}(\!\{z\}\!)$ satisfies that
	 \[
		\pr^{\le 0}(\tilde{q}_{i}(z))= q_{i}(z)
	 \]
	 for each $i=1,2,\ldots,n$,
	 and $\pr^{\le 0}\colon \mathbb{C}(\!\{z\}\!)\rightarrow \mathbb{C}[z^{-1}]$ is
	 the projection along the decomposition $\mathbb{C}(\!(z)\!)=\mathbb{C}[z^{-1}]\oplus z\mathbb{C}[\![z]\!].$
\end{prop}
\begin{proof}
	By the assumption, there exists $X\in \mathrm{GL}(n,\mathbb{C}[\![z]\!])$
	such that
	\[
	XGX^{-1}+\left(\frac{d}{dz}X\right)X^{-1}=\mathrm{diag}(q_{1}(z),q_{2}(z),\ldots,q_{n}(z))z^{-1}.
	\]
	Since $\mathrm{ord\,}\left(\frac{d}{dz}X\right)X^{-1}\ge 0$,
	it follows that
	\[
	XGX^{-1}\equiv\mathrm{diag}(q_{1}(z),q_{2}(z),\ldots,q_{n}(z))z^{-1}\quad
	(\mathrm{mod\ }\mathbb{C}[\![z]\!]).
	\]
	The regular semisimplicity assures that all $q_{i}(z)$ are mutually
	different polynomials of $z^{-1}$. Thus the result follows from same argument in Proposition \ref{Hensel}.
\end{proof}

The argument in Propositions \ref{intersection}, \ref{milnor}
and Theorem \ref{premilnor} is valid without any change even for
this case.
Thus the Milnor formula as we saw in Theorem \ref{premilnor}
 holds for this regular semisimple case.

\begin{thm}
	Let us take a differential module $M$ over $\mathbb{C}(\!\{z\}\!)$
	of rank $n$
	and a matrix $G$ of $M$ as in Proposition \ref{rss}.
	Suppose that $M$ has a singularity at $z=0$.
	Then the Milnor number of $C_{G}$ at $([\zeta:\eta],z)=([0:1],0)$ is
	\[
		\mu(C_{G})_{\infty}=-n^{2}-\mathrm{Irr}(\mathrm{End}_{\mathbb{C}(\!\{z\}\!)}(M))+2(n-1)(C_{G},\zeta)_{\infty}+(n-r_{C_{G}})+1.
	\]
\end{thm}

\section{Global comparison: Euler characteristics}
In the previous section, we have obtained a comparison formula of local singularities
of a differential module and its characteristic polynomial.
This local comparison implies the following coincidence of the
global invariants,
namely we shall show the matching of the index of rigidity of a differential
equation  and the Euler characteristic of the corresponding spectral curve.

We now come back to the differential equation
\[
	dw=Aw,\quad A\in M(n,\Omega_{X}(*D)(X)),
\]
on the Riemann surface $X$.
We use the same notation as in Section \ref{spectral curve}.
Recall that $D=\{a_{1},a_{2},\ldots,a_{p}\}$ is the set of poles of $A$.
This equation defines an algebraic connection $\nabla_{A}:=d-A$ on the trivial algebraic vector bundle
$\mathcal{O}_{U,\text{alg}}^{\oplus n}$, where $U:=X\backslash D.$
Here $\mathcal{O}_{U,\text{alg}}$ is the sheaf of regular functions on the Zariski open subset $U\subset X$.

Let $\nabla_{A}^{*}$ be the dual connection of $\nabla_{A}$, i.e.,
the dual bundle $(\mathcal{O}_{U,\text{alg}}^{\oplus n})^{*}$
with the connection
\[
\nabla_{A}^{*}(\phi)(s)=-\phi(\nabla_{A}s)
\]
where  $\phi$ and $s$ are sections of $(\mathcal{O}_{U,\text{alg}}^{\oplus n})^{*}$ and
$\mathcal{O}_{U,\text{alg}}^{\oplus n}$ respectively.

Further define the endomorphism connection as the tensor product,
$$\mathcal{E}\mathrm{nd}(\nabla_{A}):=\nabla_{A}^{*}\otimes \nabla_{A}.$$

\begin{df}[index of rigidity, Katz \cite{Kat3}]\normalfont
	The \emph{index of rigidity of $\nabla_{A}$} is the Euler characteristic
	\[
	\mathrm{rig\,}(\nabla_{A})
	=\sum_{i=0}^{2}(-1)^{i}\mathrm{dim}_{\mathbb{C}}H^{i}_{\mathrm{dR}}(X,j_{!*}(\mathcal{E}\mathrm{nd}\nabla_{A})).
	\]
	Here $j_{!*}(\mathcal{E}\mathrm{nd}\nabla_{A})$ is the middle extension by the embedding $j\colon U\hookrightarrow X$ of $\mathcal{E}\mathrm{nd}\nabla_{A}$
	 and
	$H^{*}_{\mathrm{dR}}(X,j_{!*}(\mathcal{E}\mathrm{nd}\nabla_{A}))$ are hypercohomology groups of the algebraic de Rham complex
	of $j_{!*}(\mathcal{E}\mathrm{nd}\nabla_{A})$, see II.6 in \cite{Del2} and 2.9 in \cite{Katbook}
	for more detailed treatment.
\end{df}
The Euler-Poincare formula by Deligne, Gabber and Katz, see THEOREM 2.9.9 in
\cite{Katbook}, gives a decomposition of $\mathrm{rig}(\nabla_{A})$
into a sum of local invariants as follows,
	\begin{align}\label{eulerpoincare}
	\mathrm{rig}(\nabla_{A})=(2-2g)\mathrm{rank\,}(\mathcal{E}\mathrm{nd}(\nabla_{A}))
	-\sum_{a\in D}\delta(\mathrm{End}_{\mathbb{C}(\!\{z_{a}\}\!)}(M_{A_{a}})),
	\end{align}
	where
\begin{align*}
\delta(\mathrm{End}_{\mathbb{C}(\!\{z\}\!)}(M_{A_{a}}))&:=
	\mathrm{rank\,}\mathrm{End}_{\mathbb{C}(\!\{z_{a}\}\!)}(M_{A_{a}})+\mathrm{Irr}(\mathrm{End}_{\mathbb{C}(\!\{z_{a}\}\!)}(M_{A_{a}}))\\
	 &\quad -\mathrm{dim\,}_{\mathbb{C}}\mathrm{Hom}_{\mathbb{C}(\!(z_{a})\!)}(\mathrm{End}_{\mathbb{C}(\!(z_{a})\!)}(\widetilde{M_{A_{a}}}),
	\mathbb{C}(\!(z_{a})\!))^{\text{hor}},
\end{align*}
and
$M^\text{hor}:=\{m\in M\mid \nabla_{M}m=0\}$ for a differential module $M$.

Take $a\in D$ and choose a local coordinate $(U_{i},z_{i})$ containing $a\in U_{i}$.
Then we can write $A=A_{i}dz_{i}$, $A_{i}\in M(n,\mathbb{C}(z_{i}))$ on $U_{i}$.
Let us put $z_{a}:=z_{i}-z_{i}(a)$. Then power series expansion of $A_{i}$ defines
$A_{a}\in M(n,\mathbb{C}(\!\{z_{a}\}\!))$ and a differential module
$M_{A_{a}}:=\mathbb{C}(\!\{z_{a}\}\!)^{\oplus n}$ with the derivation
$\nabla_{M_{A_{a}}}:=\frac{d}{dz_{a}}-A_{a}$.
Let us denote the point $([\zeta_{i}:\eta_{i}],z_{i})=([0:1],a)$
by $\infty_{a}$.

The following assumption enables us to apply  the results in Section \ref{HTLMILNOR} to the connection $\nabla_{A}.$
\begin{ass}\label{mfrss}\normalfont
	For each $a\in D$, the HTL normal form of
	$M_{A_{a}}$ is multiplicity free or
	$A_{a}$ is regular semisimple over $\mathbb{C}[\![z_{a}]\!]$.
\end{ass}

\begin{thm}\label{MILNER}
	Under Assumption \ref{mfrss} we have the following.
	For each $a\in D$ the Milnor number of $C_{A}$ at $\infty_{a}$
	is
	\[
	\mu(C_{A})_{\infty_{a}}=-\delta(\mathrm{End}_{\mathbb{C}(\!\{z_{a}\}\!)}(M_{A_{a}}))
	-r_{C_{A_{a}}}
	+2(n-1)(C_{A},X_{\infty})_{\infty_{a}}+1.
	\]
	Here $(C,C')_{\infty_{a}}$ is the intersection number of divisors $C, C'$
	at $\infty_{a}$.
\end{thm}
\begin{proof}
We need to  compute $\mathrm{dim\,}_{\mathbb{C}}\mathrm{Hom}_{\mathbb{C}(\!(z_{a})\!)}(\mathrm{End}_{\mathbb{C}(\!(z_{a})\!)}(\widetilde{M_{A_{a}}}),
\mathbb{C}(\!(z_{a})\!))^{\text{hor}}.$
Let
\[
	\mathrm{diag}(E_{q^{a}_{1}(z_{a}),R^{a}_{1}},\ldots,E_{q^{a}_{m_{a}}(z_{a}),R^{a}_{m_{a}}})
 \]
 be the multiplicity free HTL normal form of $M_{A_{a}}$. Then we have the irreducible
 decomposition
 \[
	\widetilde{M_{A_{a}}}\cong \bigoplus_{i=1}^{m_{a}}\widetilde{M_{E_{q^{a}_{i}(z_{a}),R^{a}_{i}}}}.
\]
This decomposition shows that
\begin{align*}
	\mathrm{Hom}_{\mathbb{C}(\!(z_{a})\!)}(\mathrm{End}_{\mathbb{C}(\!(z_{a})\!)}&(\widetilde{M_{A_{a}}}),
		\mathbb{C}(\!(z_{a})\!))^{\text{hor}}\\
		&\cong \mathrm{Hom}_{\mathbb{C}(\!(z_{a})\!)}(\widetilde{M_{A_{a}}},\widetilde{M_{A_{a}}})^{\text{hor}}\\
		&\cong
		\mathrm{Hom}_{\mathbb{C}(\!(z_{a})\!)}(\bigoplus_{i=1}^{m_{a}}\widetilde{M_{E_{q^{a}_{i}(z_{a}),R^{a}_{i}}}},\bigoplus_{i=1}^{m_{a}}\widetilde{M_{E_{q^{a}_{i}(z_{a}),R^{a}_{i}}}})^{\text{hor}}\\
		&\cong
		\bigoplus_{i=1}^{m_{a}}\bigoplus_{j=1}^{m_{a}}\mathrm{Hom}_{\mathbb{C}(\!(z_{a})\!)}(\widetilde{M_{E_{q^{a}_{i}(z_{a}),R^{a}_{i}}}},\widetilde{M_{E_{q^{a}_{j}(z_{a}),R^{a}_{j}}}})^{\text{hor}}\\
		&\cong
		\bigoplus_{i=1}^{m_{a}}\mathrm{Hom}_{\mathbb{C}(\!(z_{a})\!)}(\widetilde{M_{E_{q^{a}_{i}(z_{a}),R^{a}_{i}}}},\widetilde{M_{E_{q^{a}_{i}(z_{a}),R^{a}_{i}}}})^{\text{hor}}\\
		&\cong
		\mathbb{C}^{\oplus m_{a}}
\end{align*}
by Schur's lemma since $\widetilde{M_{E_{q^{a}_{i}(z_{a}),R^{a}_{i}}}}$ are irreducible for
$i=1,2,\ldots,m_{a}$.

Then	the desired equation directly comes from Theorem \ref{premilnor}.
\end{proof}

\begin{rem}\label{Delta}\normalfont
Let us introduce the \emph{$\delta$-invariant of a singularity} which is
defined by using the Milnor number as follows,
\[
\delta(C_{A_{a}}):=\frac{1}{2}\left(\mu(C_{A_{a}})+r_{C_{A_{a}}}-1\right).
\]
See \cite{Mil} for a geometric meaning of this invariant. Here we note
that the germ $C_{A_{a}}$ is reduced by the multiplicity free condition.
Then the above formula can be rewritten in a natural form,
\[
2\delta(C_{A_{a}})-2(n-1)(C_{A},X_{\infty})_{\infty_{a}}=-\delta(\mathrm{End}_{\mathbb{C}(\!\{z_{a}\}\!)}(M_{A_{a}})).
\]
\end{rem}

 \begin{thm}\label{maintheorem}
	We use the same notation as above. Suppose that $\nabla_{A}$
	satisfies Assumption \ref{mfrss} and  $C_{A}$ is 
	irreducible. Moreover suppose that $C_{A}$ is smooth on $T^{*}X$.
	 Then the Euler characteristic $\chi(\widetilde{C_{A}})$ of the normalization
	$\widetilde{C_{A}}$ of $C_{A}$ coincides with the index of
	rigidity of $\nabla_{A}$, i.e.,
	\[
		\chi(\widetilde{C_{A}})=\mathrm{rig}(\nabla_{A}).
	\]
 \end{thm}

 \begin{proof}
	Since $C_{A}$ is smooth on $T^{*}X$, possible singularities are
	only on $C_{A}\cap X_{\infty}=\{\infty_{a}\mid a\in D\}$.
	Hence the Euler characteristic $\chi(\widetilde{C_{A}})$ can be computed
	by the formula
	\begin{align*}
		\chi(\widetilde{C_{A}})=(2-2g_{a}(C_{A}))+2\sum_{a\in D}\delta(C_{A_{a}}),
	\end{align*}
	see Proposition 3 in section IV of  \cite{Se} for example.
We have already computed the arithmetic genus $g_{a}(C_{A})$ in the equation $(\ref{arithmetic})$. Thus
	\[
		\chi(\widetilde{C_{A}})=n^{2}(2-2g)+(2-2n)(C_{A},X_{\infty})+2\sum_{a\in D}
	\delta(C_{A_{a}}).
	\]
	Finally the formula in Remark \ref{Delta} shows that
	\begin{equation*}
	\begin{split}
	\chi(\widetilde{C_{A}})&=n^{2}(2-2g)+(2-2n)(C_{A},X_{\infty})\\
	&\quad\quad+\sum_{a\in D}\left(
	-\delta(\mathrm{End}_{\mathbb{C}(\!\{z_{a}\}\!)}(M_{A_{a}}))
	+2(n-1)(C_{A},X_{\infty})_{\infty_{a}}
	\right)\\
	&=(2-2g)\mathrm{rank}(\mathcal{E}\mathrm{nd}(\nabla_{A}))-\sum_{a\in D}
	\delta(\mathrm{End}_{\mathbb{C}(\!\{z_{a}\}\!)}(M_{A_{a}}))
	=\mathrm{rig}(\nabla_{A}).
	\end{split}
\end{equation*}
	\end{proof}
\begin{cor}\label{cohomology}
	Let $\nabla_{A}$ be as in Theorem \ref{maintheorem} and
	moreover assume that $\nabla_{A}$ is irreducible.
	Then we have the following numerical coincidences of the cohomology groups,
	\[
		h^{i}_{\mathrm{dR}}(X,j_{!*}(\mathcal{E}\mathrm{nd}\nabla_{A}))=h^{i}(\widetilde{C_{A}},\mathbb{C}),\quad
		i=0,1,2.
	\]
	Here $h^{i}(*):=\mathrm{dim}_{\mathbb{C}}H^{i}(*).$
\end{cor}
\begin{proof}
	By the irreducibility and duality, we have
	\[
		h^{0}_{\mathrm{dR}}(X,j_{!*}(\mathcal{E}\mathrm{nd}\nabla_{A}))=h^{2}_{\mathrm{dR}}(X,j_{!*}(\mathcal{E}\mathrm{nd}\nabla_{A}))=1.
	\]
	which shows that
	\[
		h^{i}_{\mathrm{dR}}(X,j_{!*}(\mathcal{E}\mathrm{nd}\nabla_{A}))=1=h^{i}(\widetilde{C_{A}},\mathbb{C})
	\]
	for $i=0,2.$
	Thus Theorem \ref{maintheorem} implies that
	\[
		h^{1}_{\mathrm{dR}}(X,j_{!*}(\mathcal{E}\mathrm{nd}\nabla_{A}))=2-\mathrm{rig\,}\nabla_{A}=2-
		\chi(\widetilde{C_{A}})=h^{1}(\widetilde{C_{A}},\mathbb{C}).
	\]
\end{proof}
\begin{rem}\normalfont
	A similar comparison of Euler characteristics of differential equations
 and another geometric counterparts, namely $\ell$-adic sheaves, has been known by Katz in \cite{Katbook}, \cite{Kat2}.
Also he pointed out a similarity between local properties, namely, the singularities of
differential equations and the ramifications of local Galois actions
on $\ell$-adic sheaves. One can find a table of analogies by Katz in  \cite{Kat}.
In this table of analogies, the irregularity of a differential equation corresponds to  the Swan conductor of the local Galois action on an $\ell$-adic
sheaf. On the arithmetic geometry side, the comparison formula
of the Swan conductor and the Milnor number has been studied,
which is called Deligne's Milnor formula, see \cite{Del}.
Our formula might give an analogy of this Milnor formula if we follow
 Katz' table.
\end{rem}

\end{document}